\documentclass[12pt,reqno,draft]{amsart}
\usepackage{enumerate, latexsym, amsmath, amsfonts, amssymb, amsthm, graphicx, color}
\def\pmod #1{\ ({\rm{mod}}\ #1)}
\def\Z{\Bbb Z}

\def\Q{\Bbb Q}

\def\C{\Bbb C}
\def\F{\Bbb F}

\def\bg{\bigg}
\def\({\bg(}
\def\){\bg)}

\def\mo{{\rm{mod}\ }}

\def\sgn{{\rm sgn}}

\def\sgn{{\rm sgn}}

\def\End{{\rm End}}

\def\ve{\varepsilon}

\def\Ack{\medskip\noindent {\bf Acknowledgments}}
\theoremstyle{plain}
\newtheorem{theorem}{Theorem}

\newtheorem{lemma}{Lemma}
\newtheorem{corollary}{Corollary}
\newtheorem{conjecture}{Conjecture}
\theoremstyle{definition}

\theoremstyle{remark}
\newtheorem{remark}{Remark}

\makeatletter
\@namedef{subjclassname@2010}{%
  \textup{2010} Mathematics Subject Classification}
\makeatother
 \vspace{4mm}

\begin{document}
 \baselineskip=17pt
\hbox{} {}
\medskip
\title[Elliptic curves over $\mathbb{F}_p$ and determinants of Legendre matrices]
{Elliptic curves over $\mathbb{F}_p$ and determinants of Legendre matrices}
\date{}
\author[Hai-Liang Wu] {Hai-Liang Wu}

\thanks{
\newline \indent Supported by the National Natural Science
Foundation of China (Grant No. 11971222).}

\address {(Hai-Liang Wu) School of Science, Nanjing University of Posts and Telecommunications, Nanjing 210023, People's Republic of China}
\email{\tt whl.math@smail.nju.edu.cn}

\begin{abstract}
Determinants with Legendre symbol entries have close relations with character sums and elliptic curves over finite fields. In recent years, Sun \cite{Sun2019det}, Krachun and his cooperators \cite{Dimitry} studied this topic. In this paper, we confirm some conjectures posed by Sun and investigate some related topics.

For instance, given any integers $c,d$ with $d\ne0$ and $c^2-4d\ne0$, we show that there are infinitely many odd primes $p$ such that
$$\det\bigg[\left(\frac{i^2+cij+dj^2}{p}\right)\bigg]_{0\le i,j\le p-1}=0,$$
where $(\frac{\cdot}{p})$ is the Legendre symbol. This confirms a conjecture of Sun.
\end{abstract}
\maketitle
\section{Introduction}
\setcounter{lemma}{0}
\setcounter{theorem}{0}
\setcounter{corollary}{0}
\setcounter{remark}{0}
\setcounter{equation}{0}
\setcounter{conjecture}{0}
\setcounter{proposition}{0}
Investigating determinants concerning Legendre symbols is an active topic in number theory and finite fields. This topic has close relations with character sums, quadratic residues and elliptic curves. Throughout this paper, for any matrix $M=[a_{ij}]_{1\le i,j\le n}$ $(n\in\Z^+)$, we use $\det M$ or $|M|$ to denote the determinant of $M$.
\subsection{Classical Results}
Let $p$ be an odd prime and let $(\frac{\cdot }{p})$ denote the Legendre symbol. Given any homogeneous polynomial $H(X,Y)\in\Z[X,Y]$ and any positive integer $n$, the determinant $\det[(\frac{H(i,j)}{p})]_{1\le i,j\le n}$ was studied by many mathematicians. In the case $\deg(H)=1$ (where $\deg(H)$ denotes the degree of $H$), Carlitz \cite{carlitz} studied the following determinant
$$\det\bigg[c+\(\frac{i-j}{p}\)\bigg]_{1\le i,j\le p-1}\ \ \ \ \ (c\in\C).$$
He \cite[Theorem 4 (4.9)]{carlitz} showed that the characteristic polynomial of the matrix $[c+(\frac{i-j}{p})]_{1\le i,j\le p-1}$ is
$$f_{\chi,c}(t)=(t^2-(-1)^{(p-1)/2}p)^{(p-3)/2}(t^2-(p-1)c-(-1)^{(p-1)/2}).$$
Hence $\det[(\frac{i-j}{p})]_{1\le i,j\le p-1}=(-1)^{p-1}f_{\chi,0}(0)=p^{\frac{p-3}{2}}$. Along this line, Chapman \cite{chapman} studied the following two determinants:
$$C_1:=\det\left[\(\frac{i+j-1}{p}\)\right]_{1\le i,j\le\frac{p-1}{2}}$$
and
$$C_2:=\det\left[\(\frac{i+j-1}{p}\)\right]_{1\le i,j\le\frac{p+1}{2}}.$$
Noting that $\frac{p+1}{2}-i+\frac{p+1}{2}-j\equiv-(i+j-1)\pmod p$, we obtain that $C_1$ is equal to
\begin{align*}
\det\bigg[\(\frac{\frac{p+1}{2}-i+\frac{p+1}{2}-j-1}{p}\)\bigg]_{1\le i,j\le\frac{p-1}{2}}
=\(\frac{-1}{p}\)\det\bigg[\(\frac{i+j}{p}\)\bigg]_{1\le i,j\le\frac{p-1}{2}}.
\end{align*}
With the same reason, we also have
$$C_2=\det\bigg[\(\frac{i+j}{p}\)\bigg]_{0\le i,j\le\frac{p-1}{2}}.$$
By using quadratic Gauss sums, Chapman obtained the explicit values of them. In fact,
when $p\equiv 1\pmod 4$, let $\varepsilon_p>1$ and $h(p)$ denote the fundamental unit and class number of the real quadratic field $\mathbb{Q}(\sqrt{p})$ respectively and let $\varepsilon_p^{h(p)}=a_p+b_p\sqrt{p}$ with $2a_p,2b_p\in\Z$. Then
 	$$C_1=
 \begin{cases}
 (-1)^{(p-1)/4}2^{(p-1)/2}b_p   &\mbox{if}\ p \equiv 1 \pmod 4,\\
 0 &\mbox{if} \ p\equiv 3 \pmod 4,
 \end{cases}
 $$	
and
	$$C_2=
\begin{cases}
 (-1)^{(p+3)/4}2^{(p-1)/2}a_p   &\mbox{if}\ p \equiv 1 \pmod 4,\\
-2^{(p-1)/2} &\mbox{if} \ p\equiv 3 \pmod 4.
\end{cases}
$$	
Keep the above notations and write
$$\ve_p^{(2-(\frac{2}{p}))h(p)}=a_p'+b_p'\sqrt{p},\ \text{with}\ a_p',b_p'\in\Q.$$
Chapman \cite{evil} posed a conjecture concerning the following determinant
$$C_3:=\det\bigg[\(\frac{j-i}{p}\)\bigg]_{1\le i,j\le\frac{p+1}{2}},$$
which says that
$$
C_3=\begin{cases}
 -a_p'   &\mbox{if}\ p \equiv 1 \pmod 4,\\
1 &\mbox{if} \ p\equiv 3 \pmod 4.
\end{cases}
$$
Due to the difficulty of the evaluation on this determinant, he even called this determinant ``evil" determinant. By using an elegant matrix decomposition, Vsemirnov \cite{M1,M2} confirmed this conjecture completely (case $p\equiv3\pmod4$ in \cite{M1} and case $p\equiv1\pmod4$ in \cite{M2}).
\subsection{Sun's determinant $T_p$}
Inspired by the above works, Sun \cite{Sun2019det} studied the case $\deg(H)=2$. For example, Sun considered the determinant:
\begin{equation}\label{Eq. S(1,p)}
S(1,p):=\det\bigg[\(\frac{i^2+j^2}{p}\)\bigg]_{1\le i,j\le\frac{p-1}{2}},
\end{equation}
and showed that $-S(1,p)$ is always a quadratic residue modulo $p$. Inspired by this result, Sun \cite[Conjecture 4.4]{Sun2019det} made the following conjecture:
\begin{conjecture}\label{Conj. triangular numbers}{\rm (Sun)}
Let $p\equiv\pm3\pmod8$ be a prime. Then $-T_p$ is a quadratic non-residue modulo $p$,
where
$$T_p:=\det\bigg[\(\frac{i(i+1)+j(j+1)}{p}\)\bigg]_{1\le i,j\le\frac{p-1}{2}}.$$
\end{conjecture}
For example, $T_5=-2,T_{11}=4,T_{13}=-8,T_{19}=928\equiv 16\pmod{19}$ and $T_{23}=-6656\equiv-9\pmod{23}$.

As the first result of this paper, we confirm this conjecture and obtain the following general result:
\begin{theorem}\label{Thm. triangular numbers}
Let $p$ be an odd prime. Then the following results hold.

{\rm (i)} If $p\equiv\pm3\pmod8$, then we have $(\frac{T_p}{p})=(-1)^{\frac{p-3}{2}}$.

{\rm (ii)} If $p\equiv\pm1\pmod8$, then we have $(\frac{T_p}{p})=(\frac{-3}{p})$.
\end{theorem}
\subsection{The determinants $(c,d)_p$ and $[c,d]_p$}
Given an odd prime $p$ and integers $c,d$, Sun \cite{Sun2019det} also studied the following determinants:
\begin{equation}\label{Eq. (c,d)p}
(c,d)_p:=\det\bigg[\(\frac{i^2+cij+dj^2}{p}\)\bigg]_{1\le i,j\le p-1},
\end{equation}
and
\begin{equation}\label{Eq. [c,d]p}
[c,d]_p:=\det\bigg[\(\frac{i^2+cij+dj^2}{p}\)\bigg]_{0\le i,j\le p-1}.
\end{equation}
For example, Sun showed that $(c,d)_p=0$ whenever $d$ is a quadratic non-residue modulo $p$. These determinants were later investigated by Krachun, Petrov, Sun, and Vsemirnov \cite{Dimitry}. For example, they proved that $(6,1)_p=[6,1]_p=(3,2)_p=[3,2]_p=0$ for any primes $p\equiv3\pmod4$. This confirms a conjecture of Sun \cite[Conjecture 4.8(ii)]{Sun2019det}. In this paper, we study some properties of these determinants with the help of character sums and elliptic curves over finite fields. Sun \cite[Conjecture 4.7(i)]{Sun2019det} made the following conjecture.
\begin{conjecture}\label{Conj. [c,d]p}{\rm (Sun)}
If $d$ is non-zero and $c^2-4d\ne0$, then $[c,d]_p=0$ for infinitely many primes $p$.
\end{conjecture}
By using some theory of supersingular elliptic curves we confirm this conjecture and obtain the following result.
\begin{theorem}\label{Thm. [c,d]p}
{\rm (i)} Given any integers $c,d$ with $d\ne0$ and $c^2-4d=0$, we have
$[c,d]_p=(\frac{-2c}{p})\times(p-1)$ for any odd prime $p$.

{\rm (ii)} Given any integers $c,d$ with $d\ne0$ and $c^2-4d\ne0$, there are infinitely many primes $p$ such that $[c,d]_p=0$. Moreover, let $p$ be an odd prime with $p\nmid c^2-4d$, and let $A_{c,d}(p)$ denote the coefficient of $x^{\frac{p-1}{2}}$ of the polynomial $(dx^2+cx+1)^{\frac{p-1}{2}}$. Then $[c,d]_p=0$ if $p\mid A_{c,d}(p)$.
\end{theorem}
\begin{remark}\label{Remark A}
(i) $A_{c,d}(p)$ is usually called the generalized central trinomial coefficient. Readers may refer to \cite{Sun201401} for a detailed introduction on this topic.

(ii) $A_{c,d}(p)$ also has close relations with the number of rational points over a curve in $\F_p$. Let $E_{c,d}$ be a curve over $\F_p$ defined by
$$E_{c,d}:=\{(x,y)\in\F_p\times\F_p:\ y^2=x(dx^2+cx+1)\}\cup\{\infty\},$$
and let $\#E_{c,d}=p+1-a_{c,d}(p)$, where $\#S$ denotes the cardinality of a set $S$. We claim that $a_{c,d}(p)\equiv A_{c,d}(p)\pmod p$. In fact, one can verify that
$$\#E_{c,d}(p)=1+\sum_{x=0}^{p-1}\(1+\(\frac{x(dx^2+cx+1)}{p}\)\).$$
This gives
$$a_{c,d}(p)=-\sum_{x=0}^{p-1}\(\frac{x(dx^2+cx+1)}{p}\).$$
On the other hand, for any integer $k$ we have (cf. \cite[p. 235, Lemma 2]{classical})
$$
\sum_{x=0}^{p-1}x^k\equiv\begin{cases}
0\ {\rm mod}\ p   &\mbox{if}\ p-1\nmid k,\\
-1\ {\rm mod}\ p &\mbox{if} \ p-1\mid k.
\end{cases}
$$
Hence by Euler's criterion
\begin{align*}
a_{c,d}(p)=-\sum_{x=0}^{p-1}\(\frac{x(dx^2+cx+1)}{p}\)
&\equiv-\sum_{x=0}^{p-1}x^{\frac{p-1}{2}}(dx^2+cx+1)^{\frac{p-1}{2}}\\
&\equiv-\sum_{x=0}^{p-1}A_{c,d}(p)x^{p-1}\equiv A_{c,d}(p)\pmod p.
\end{align*}
\end{remark}
To make the above result explicit, we give an example here. Consider the determinant $[6,1]_p$ (where $p>3$ is a prime). By using some elegant congruences, Krachun and his cooperators \cite[Corollary 1.1]{Dimitry} showed that $[6,1]_p=0$ if $p\equiv3\pmod4$. Here we give another proof by using Theorem \ref{Thm. [c,d]p}. Let $E_{6,1}$ be the elliptic curve over $\Q$ with complex multiplication by the imaginary quadratic field $\Q(\sqrt{-1})$ defined by the equation: $y^2=x^3+6x^2+x$. We let
$$\#E_{6,1}(\F_p)=1+\#\{(x,y)\in\F_p^2: y^2=x^3+6x^2+x\}=p+1-a_{6,1}(p).$$
As $E_{6,1}$ has complex multiplication by $\Q(\sqrt{-1})$, it is known that $a_{6,1}(p)=0$ if $(\frac{-1}{p})=-1$. By (ii) of Remark \ref{Remark A} we have $A_{6,1}(p)\equiv a_{6,1}(p)\equiv0\pmod p$.
This implies that $[6,1]_p=0$ if $p\equiv3\pmod4$ and $p>3$.

Now we turn back to the determinant $S(1,p)$ defined in (\ref{Eq. S(1,p)}). In 2019, Sun conjectured that $-S(1,p)$ is an integral square if $p\equiv3\pmod4$. By using an elegant matrix decomposition,
Alekseyev and Krachun confirmed this conjecture. In the case $p\equiv1\pmod4$, if we write $p=a^2+4b^2$ with $a,b\in\Z$ and $a\equiv1\pmod4$, then Cohen, Sun and Vsemirnov conjectured that $S(1,p)/a$ is an integral square (see \cite[Remark 4.2]{Sun2019det}). This conjecture was later proved by the author \cite{Wu}. In this paper, we generalize the above results to determinants $(c,d)_p$ with $p\nmid d$ defined in (\ref{Eq. (c,d)p}). Moreover, we claim that it suffices to study the case $d=1$. In fact,
Sun \cite[Theorem 1.3(i)]{Sun2019det} showed that $(c,d)_p=0$ if $(\frac{d}{p})=-1$. Now we suppose that $d\equiv f^2\pmod p$, and let $g\in\Z$ with $g\equiv c/f\pmod p$. Then by Zolotarev's lemma \cite{Zolotarev} (see also Lemma \ref{Lem. Z lemma} in Section 3) one can verify that
\begin{equation}\label{Eq. introduction}
(c,d)_p=\(\frac{f}{p}\)\times(g,1)_p.
\end{equation}
Hence we just need to consider the determinants $(c,1)_p$.
Let $p>3$ be an odd prime. When $c\not\equiv\pm2\pmod p$, clearly the equation $y^2=x(x^2+cx+1)$ defines an elliptic curve $E_c(\F_p)$ over the finite field $\F_p$, i.e.,
$$E_c(\F_p)=\{(x,y)\in\F_p\times\F_p: y^2=x(x^2+cx+1)\}\cup\{\infty\}$$
is an elliptic curve over $\F_p$. With the standard notation, we set $\#E_c(\F_p)=p+1-a_p(c)$. It is known that $a_p(c)$ can be interpreted as the trace of the Frobenius endomorphism (see \cite[p. 144]{Silverman}). Now we state our result.
\begin{theorem}\label{Thm. (c,1)p}
Let $p>3$ be an odd prime. Then the following results hold.

{\rm (i)} If $c\equiv\pm2\pmod p$, then we have $(c,1)_p=(\frac{-2c}{p})\times(2-p)$.

{\rm (ii)} If $c\not\equiv\pm2\pmod p$, then $(c,1)_p=2a_p(c)x_p(c)^2$ for some integer $x_p(c)$.
\end{theorem}

Sun \cite[Conjecture 4.8]{Sun2019det} posed the following conjectures:
\begin{conjecture}\label{Conj. (c,d)p and squares}{\rm (Sun)}
{\rm (i)} Let $p\equiv1\pmod4$ be a prime with $p=x^2+4y^2$ $(x,y\in\Z$ and $4\mid x-1)$. Then $(3,2)_p=(-1)^{\frac{p-1}{4}}xv^2$ for some integer $v$.

{\rm (ii)} Let $p\equiv1\pmod8$ be a prime with $p=x^2+2y^2$ with $(x,y\in\Z$ and $4\mid x-1)$. Then
$(4,2)_p=(8,8)_p=(-1)^{\frac{p-1}{8}}xw^2$ for some integer $w$. Also, let $p\equiv1\pmod{12}$ be a prime with $p=x^2+3y^2$ $(x,y\in\Z$ and $3\mid x-1)$. Then $(3,3)_p/x$ is an integral square.

{\rm (iii)} Let $p\equiv 1,9,25\pmod{28}$ with $p=x^2+7y^2$ $(x,y\in\Z$ and $(\frac{x}{7})=1)$. Then
$(21,112)_p/x$ is an integral square.
\end{conjecture}

With the help of (\ref{Eq. introduction}) and Theorem \ref{Thm. (c,1)p} we can confirm these conjectures.
\begin{corollary}\label{Cor. squares and determinants}
Conjecture \ref{Conj. (c,d)p and squares} holds.
\end{corollary}
For example, if $p=41=5^2+4\times 2^2$, then $(3,2)_{41}/5=17951494350240^2$.
\subsection{The determinants $W_p$ and $Y_p$}
Motivated by Theorem \ref{Thm. (c,1)p}, we also consider some special cases in which $\deg(H)>2$. Observe first that for any integer $k>1$ with $k\mid p-1$, we always have $\det[(\frac{i^k+j^k}{p})]_{1\le i,j\le p-1}=0$. Hence we first consider the following variant.

Let $p\equiv1\pmod4$ be a prime with $p=x_1^2+4y_1^2$ $(x_1,y_1\in\Z$ and $x_1\equiv1\pmod4)$.
Let $0<a_1<a_2<\cdots<a_{\frac{p-1}{4}}<p$ be all the biquadratic residues modulo $p$ in the interval $(0,p)$. We consider the following determinant:
\begin{equation}\label{Eq. definition of Wp biquadratic}
W_p:=\det\bigg[\(\frac{a_i+a_j}{p}\)\bigg]_{1\le i,j\le\frac{p-1}{4}}.
\end{equation}
In addition, if $p\equiv1\pmod8$, then there are integers $x_2,y_2$ with $x_2\equiv 1\pmod4$ such that $p=x_2^2+2y_2^2$ (cf. \cite[Chapter 1]{Cox}). With these notations, we obtain the following result:
\begin{theorem}\label{Thm. biquadratic}
Let $p\equiv1\pmod4$ be a prime. Then the following results hold.

{\rm (i)} If $p\equiv5\pmod8$, then $-2W_p/(1+x_1)$ is an integral square.

{\rm (ii)} If $p\equiv1\pmod8$, then $(-1)^{\frac{p-1}{8}}2W_p/((1+x_1)x_2)$ is an integral square.
\end{theorem}
For example, (i) when $p=13=(-3)^2+4\times1^2$, we have $a_1=1,a_2=3,a_3=9, W_{13}=4$ and $-2W_{13}/(1+(-3))=4$ is a square.

(ii) When $p=17=1^2+4\times2^2=(-3)^2+2\times2^2$, we have $a_1=1,a_2=4,a_3=13,a_4=16, W_{17}=-3$ and
$(-1)^{\frac{17-1}{8}}2W_{17}/((1+1)(-3))=1$ is a square.

Now we consider another variant which is related to the curve defined by $y^2=x^3+1$ over $\F_p$.

Let $p\equiv1\pmod6$ be a prime, and let $0<c_1<\cdots<c_{\frac{p-1}{6}}<p$ be all the $6$-th power residues modulo $p$ in the interval $(0,p)$. We consider the determinant:
\begin{equation}\label{Eq. definition of Yp six}
Y_p:=\det\bigg[\(\frac{c_i+c_j}{p}\)\bigg]_{1\le i,j\le\frac{p-1}{6}}.
\end{equation}
It is known that there are integers $x_3,y_3$ with $x_3\equiv1\pmod3$ such that $p=x_3^2+3y_3^2$. In addition, if $p\equiv1\pmod{12}$, then we write $p=x_4^2+4y_4^2$ with $x_4,y_4\in\Z$ and $x_4\equiv (\frac{2}{p})\pmod 4$. With these notations, we obtain the following result:
\begin{theorem}\label{Thm. six}
Let $p\equiv1\pmod6$ be a prime. Then the following results hold.

{\rm (i)} If $p\equiv7\pmod{12}$, then $-3Y_p/(1+2x_3)$ is an integral square.

{\rm (ii)} If $p\equiv1\pmod{12}$, then $(-1)^{\frac{p+3}{4}}3Y_p=(1+2x_3)x_4\delta_pz_p^2$ for some integer $z_p$, where
$$\delta_p:=\begin{cases}1/3&\mbox{if}\ 3\mid x_4,\\-1&\mbox{if}\ 3\nmid x_4.\end{cases}$$
\end{theorem}
For instance, (i) when $p=31=(-2)^2+3\times3^2$, we have $c_1=1,c_2=2,c_3=4,c_4=8,c_5=16,Y_{31}=16$ and $-3Y_{31}/(1+2\times(-2))=16$.

(ii) when $p=13=1^2+3\times2^2=3^2+4\times1^2$, we have $c_1=1,c_2=12, Y_{13}=1$ and
$(-1)^{\frac{13+3}{4}}\times3\times1=(1+2\times1)\times3\times\frac{1}{3}\times1^2$.

The proof of Theorem \ref{Thm. triangular numbers} will be given in Section 2. We will prove Theorem \ref{Thm. [c,d]p}--\ref{Thm. (c,1)p} in Section 3. Finally, the proofs of Theorems \ref{Thm. biquadratic}--\ref{Thm. six} will be given in Section 4.
\maketitle
\section{Proof of Theorem \ref{Thm. triangular numbers}}
\setcounter{lemma}{0}
\setcounter{theorem}{0}
\setcounter{corollary}{0}
\setcounter{remark}{0}
\setcounter{equation}{0}
\setcounter{conjecture}{0}
We first introduce some notations. Let $X_1,\cdots,X_m$ be variables. For each $1\le k\le m$, the $k$-th elementary symmetric polynomial $\sigma_k$ with respect to $X_1,\cdots,X_m$ is defined by
$$\sigma_k(X_1,\cdots,X_m):=\sum_{1\le i_1<\cdots<i_k\le m}\prod_{j=1}^kX_{i_j}.$$
Also, we let $\sigma_0(X_1,\cdots,X_m)=1$.

Let $m$ be a positive integer, and let $x_1,\cdots,x_m, y_1,\cdots,y_m$ be complex numbers. It is known that (cf. \cite[Lemma 9]{Kratt}) for a complex polynomial $P(X)=a_{m-1}X^{m-1}+\cdots+a_0$ we have
$$\det[P(x_i+y_j)]_{1\le i,j\le m}
=a_{m-1}^m\prod_{k=0}^{m-1}\binom{m-1}{k}\times\prod_{1\le i<j\le m}(x_i-x_j)(y_j-y_i).$$
Grinberg, Sun and Zhao \cite[Theorem 3.1]{Gringerg} studied a variant of this result.
Let $G_m:=\det[(x_i+y_j)^m]_{1\le i,j\le m}$. Grinberg, Sun and Zhao \cite[Theorem 3.1]{Gringerg} proved the following useful result:
\begin{lemma}\label{Lem. Grinberg}{\rm(Grinberg, Sun and Zhao)}
$$G_m=(-1)^{\frac{m(m-1)}{2}}\times\prod_{1\le i<j\le m}(x_j-x_i)(y_j-y_i)\times\prod_{r=0}^m\binom{m}{r}\times S_m,$$
where
$$S_m:=\sum_{k=0}^m\frac{\sigma_k(x_1,\cdots,x_m)\sigma_{m-k}(y_1,\cdots,y_m)}{\binom{m}{k}}.$$
\end{lemma}

We also need the following result.
\begin{lemma}\label{Lem. congruence of polynomials}
Let $p=2n+1$ be an odd prime. Then we have
$$\prod_{r=1}^n(X-r^2-r)\equiv\frac{(4X+1)^n-1}{4X}(4X+1)\pmod {p\Z[X]}.$$
Moreover, we have
$$\sigma_n(1^2+1,\cdots,n^2+n)\equiv (-1)^nn\pmod p,$$
and
$$\sigma_k(1^2+1,\cdots,n^2+n)\equiv (-1)^k4^{n-k}\binom{n+1}{k}\pmod p$$
for $0\le k\le n-1$.
\end{lemma}
\begin{proof}
For convenience, we let
$$F(X):=\prod_{r=1}^n(X-r^2-r)\ \mo p\Z[X],$$
and let
$$G(X):=\frac{(4X+1)^n-1}{4X}(4X+1)\ \mo p\Z[X].$$
Clearly we can view $F(X)$ and $G(X)$ as two polynomials in $(\Z/p\Z)[X]=\F_p[X]$. Since $\deg(F)=\deg(G)=n$ and $F,G$ have the same leading coefficients (because $4^n=2^{p-1}\equiv1\pmod p$), it suffices to show that $r^2+r\ \mo p$ $(r=1,\cdots,n)$ are exactly all roots of $G(X)$. One can easily verify that $1^2+1\ \mo p,\cdots, n^2+n\ \mo p$ are distinct and $G(r^2+r\ \mo p)=0\ \mo p$ for each $r=1,\cdots,n$. Hence we obtain
\begin{equation}\label{Eq. congruence of polynomial}
\prod_{r=1}^n(X-r^2-r)\equiv\frac{(4X+1)^n-1}{4X}(4X+1)\pmod {p\Z[X]}.
\end{equation}
Note that
$$\prod_{r=1}^n(X-r^2-r)=\sum_{k=0}^n(-1)^k\sigma_k(1^2+1,\cdots,n^2+n)X^{n-k}$$
and
\begin{align*}
\frac{(4X+1)^n-1}{4X}(4X+1)&=\frac{1}{4X}\((4X+1)^{n+1}-(4X+1)\)\\
&=\sum_{k=0}^{n-1}\binom{n+1}{k}(4X)^{n-k}+(n+1)-1\\
&=(4X)^n+\sum_{k=1}^{n-1}\binom{n+1}{k}(4X)^{n-k}+n.
\end{align*}
By (\ref{Eq. congruence of polynomial}) one can get the desired result.
\end{proof}
We are now in a position to prove Theorem \ref{Thm. triangular numbers}.

{\bf Proof of Theorem \ref{Thm. triangular numbers}.} In this proof, we assume that $p>3$ and $p=2n+1$. Euler's criterion says that $(\frac{a}{p})\equiv a^n\pmod p$ for any $a\in\Z$. Thus,
$$T_p\equiv \det[(i^2+i+j^2+j)^n]_{1\le i,j\le n}\pmod p.$$
By Lemma \ref{Lem. Grinberg} we thus have
\begin{equation*}
T_p\equiv (-1)^{\frac{n(n-1)}{2}}\times\prod_{1\le i<j\le n}(j^2+j-i^2-i)^2\times\prod_{r=0}^n\binom{n}{r}\times S_n\pmod p,
\end{equation*}
where
$$S_n=\sum_{k=0}^n\frac{\sigma_k(1^2+1,\cdots,n^2+n)\sigma_{n-k}(1^2+1,\cdots,n^2+n)}{\binom{n}{k}}.$$
Hence we obtain
\begin{equation}\label{Eq. A in Thm. triangular numbers}
\(\frac{T_p}{p}\)
=\(\frac{-1}{p}\)^{n(n-1)/2}\cdot\(\frac{\prod_{r=0}^n\binom{n}{r}}{p}\)\cdot\(\frac{S_n}{p}\).
\end{equation}
We first consider $S_n\ \mo p$. By Lemma \ref{Lem. congruence of polynomials} we have
\begin{align*}
S_n&\equiv (-1)^n2n+\sum_{k=1}^{n-1}(-1)^n4^n\binom{n+1}{k}\binom{n+1}{n-k}\bigg/\binom{n}{k}\\
&\equiv (-1)^{n+1}+(-1)^n\sum_{k=1}^{n-1}\binom{n+1}{k}\binom{n+1}{k+1}\bigg/\binom{n}{k}\\
&\equiv (-1)^{n+1}+(-1)^n\sum_{k=1}^{n-1}\frac{n+1}{k+1}\binom{n+1}{k}\pmod p.
\end{align*}
The last congruence follows from
$$\binom{n+1}{k+1}=\frac{n+1}{k+1}\binom{n}{k}.$$
By the identity:
$$\sum_{k=0}^{n+1}\frac{1}{k+1}\binom{n+1}{k}=\frac{2^{n+2}-1}{n+2},$$
we obtain
\begin{align*}
S_n\equiv (-1)^{n+1}+(-1)^n(n+1)\(\frac{2^{n+2}-1}{n+2}-1-1-\frac{1}{n+2}\)\pmod p.
\end{align*}
From this one can verify that
\begin{equation}\label{Eq. B in Thm. triangular numbers}
S_n\equiv\begin{cases}4\times(-1)^{n+1}\pmod p&\mbox{if}\ p\equiv\pm3\pmod8,\\\frac{4}{3}\times(-1)^{n+1}\pmod p&\mbox{if}\ p\equiv
\pm1\pmod8.\end{cases}
\end{equation}
Now we consider the product $\prod_{r=0}^{n}\binom{n}{r}$. It is easy to see that
$$\prod_{r=0}^{n}\binom{n}{r}=\frac{n!^{n+1}}{(0!1!\cdots n!)^2}.$$
Moreover, Sun \cite[Lemma 2.3]{Sun2019det} showed that $(\frac{n!}{p})=(\frac{2}{p})$ if $n$ is even.
Thus
\begin{equation}\label{Eq. C in Thm. triangular numbers}
\(\frac{\prod_{r=0}^{n}\binom{n}{r}}{p}\)=\(\frac{n!}{p}\)^{n+1}=\begin{cases}1&\mbox{if}\ p\equiv 3\pmod4,\\(\frac{2}{p})&\mbox{if}\ p\equiv
1\pmod4.\end{cases}
\end{equation}
In view of (\ref{Eq. A in Thm. triangular numbers})--(\ref{Eq. C in Thm. triangular numbers}), one can easily get the desired result.\qed
\maketitle
\section{Proofs of Theorems \ref{Thm. [c,d]p}--\ref{Thm. (c,1)p}}
\setcounter{lemma}{0}
\setcounter{theorem}{0}
\setcounter{corollary}{0}
\setcounter{remark}{0}
\setcounter{equation}{0}
\setcounter{conjecture}{0}
We first give a brief introduction on elliptic curves over finite fields. Readers may refer to \cite[Chapter V]{Silverman} for details. Let $p>3$ be a prime, and let $\F_p$ be the finite field with $p$ elements. Let $E$ be an elliptic curve over $\F_p$ defined by the equation
$$y^2=f(x),$$
where $f(x)\in\F_p[x]$ with $\deg(f)=3$. We set
$$E(\F_p):=\{(x,y)\in\F_p\times\F_p: y^2=f(x)\}\cup\{\infty\}.$$
As mentioned in the introduction, we set
\begin{equation}\label{Eq. ap}
\#E(\F_p)=p+1-a_p.
\end{equation}
It is clear that
$$a_p=-\sum_{x=0}^{p-1}\(\frac{f(x)}{p}\).$$
In addition, we say that $E$ is {\it supersingular} if the endomorphism ring $\End(E)$ is an order in a quaternion algebra (readers may refer to \cite[pp. 144-145]{Silverman} for details). There are many known criteria for an elliptic curve over finite field to be supersingular (cf. \cite[Table 2, p. 269]{Husemoller}). For our purpose, we need the following result.

\begin{lemma}\label{Lem. criterion for supersingular}
Let $p>3$ be a prime, and let $E$ be an elliptic curve over $\F_p$. Then $E$ is supersingular if and only if $\#E(\F_p)=p+1$, i.e., $a_p=0$.
\end{lemma}
There are many works on supersingular elliptic curves. For example, Elkies \cite{Elkies} showed the following elegant result:
\begin{lemma}\label{Lem. infinitely many supersingular}{\rm (Elkies)}
For any degree-$3$ polynomial $f\in\Z[x]$ without multiple roots, there exist infinitely many primes $p$ such that the curve over $\F_p$ defined by $y^2=f(x)$ is a supersingular elliptic curve over $\F_p$.
\end{lemma}
We also need the following result.

Let $p$ be an odd prime, and let $\F_p=\Z/p\Z=\{0\ \mo p,\cdots,p-1\ \mo p\}$. Given any $a,b\in\Z$ with $p\nmid ab$, clearly
$$0\ \mo p, 1\times a/b\ \mo p,\cdots, (p-1)\times a/b\ \mo p$$
is a permutation $\pi_p(a/b)$ of the sequence
$$0\ \mo p, 1\ \mo p,\cdots, p-1\ \mo p.$$
We also let $\sgn(\pi_p(a/b))$ denote the sign of $\pi_p(a/b)$.
In the case $b=1$, Zolotarev's lemma \cite{Zolotarev} says that $\sgn(\pi_p(a))=(\frac{a}{p})$. However, it is easy to generalize this to any $b\in\Z$ with $p\nmid b$. In fact, let $b'$ be an integer with $bb'\equiv1\pmod p$. Then clearly $(\frac{b}{p})=(\frac{b'}{p})$. Note that $\pi_p(a/b)=\pi_p(ab')$. By Zolotarev's lemma \cite{Zolotarev} we obtain that $$\sgn(\pi_p(a/b))=\sgn(\pi_p(ab'))=\(\frac{ab'}{p}\)=\(\frac{ab}{p}\).$$
In view of the above, we obtain the following result:
\begin{lemma}\label{Lem. Z lemma}{\rm (Zolotarev)}
Let notations be as the above. Then
$$\sgn(\pi_p(a/b))=\(\frac{ab}{p}\).$$
\end{lemma}
Now we are in a position to prove Theorem \ref{Thm. [c,d]p}. As usual, given any matrix $M$, the symbol $M^T$ denotes the transpose of $M$.

{\bf Proof of Theorem \ref{Thm. [c,d]p}.} In this proof, for any odd prime $p$ we set
$$M_p:=\bigg[\(\frac{i^2+cij+dj^2}{p}\)\bigg]_{0\le i,j\le p-1}.$$
Suppose first that $c^2-4d=0$ and $p\mid d$. Then clearly $M_p$ is singular. Suppose now that $c^2-4d=0$ and $p\nmid d$. Then we have
$$M_p=\bigg[\(\frac{(i+cj/2)^2}{p}\)\bigg]_{0\le i,j\le p-1}.$$
By Lemma \ref{Lem. Z lemma} one can easily verify that
$$\det M_p=\sgn(\pi_p(-2/c))\times\left| \begin{array}{cccccccc}
 0  & 1  & \ldots  & 1  \\
 1 & 0 &  \ldots  &  1\\
 \vdots & \vdots & \ddots  & \vdots    \\
 1& 1& \ldots &  0\\
 \end{array} \right|=\(\frac{-2c}{p}\)\times(p-1).$$
This proves (i).

Now we consider the case $c^2-4d\ne0$. We claim that
$$\mu_p:=\sum_{j=1}^{p-1}\(\frac{1+cj+dj^2}{p}\)\(\frac{j}{p}\)$$
is an eigenvalue of $M_p$, and that
$${\bf v}_p:=\(0,\(\frac{1}{p}\),\cdots,\(\frac{p-1}{p}\)\)^T$$
is an eigenvector for $\mu_p$. In fact, for each $1\le i\le p-1$, we have
\begin{align*}
\sum_{j=1}^{p-1}\(\frac{i^2+cij+dj^2}{p}\)\(\frac{j}{p}\)
&=\sum_{j=1}^{p-1}\(\frac{1+cj/i+dj^2/i^2}{p}\)\(\frac{j/i}{p}\)\(\frac{i}{p}\)\\
&=\sum_{j=1}^{p-1}\(\frac{1+cj+dj^2}{p}\)\(\frac{j}{p}\)\(\frac{i}{p}\)=\mu_p\(\frac{i}{p}\).
\end{align*}
For $i=0$, we have
$$\sum_{j=1}^{p-1}\(\frac{d}{p}\)\(\frac{j}{p}\)=0.$$
In view of the above, we have $M_p{\bf v}_p=\mu_p{\bf v}_p$.

As $d\ne0$ and $c^2-4d\ne0$, clearly the equation $y^2=dx^3+cx^2+x$ defines an elliptic curve $E$ over $\Q$. By Lemma \ref{Lem. infinitely many supersingular} we know that there are infinitely many primes $p>3$ such that
$$E_p:=\{(x,y)\in\F_p\times\F_p: y^2=dx^3+cx^2+x\}\cup\{\infty\}$$
is a supersingular elliptic curve over $\F_p$. By Lemma \ref{Lem. criterion for supersingular} we have $\#E_p=p+1$. By (\ref{Eq. ap}) we have
$$\sum_{j=0}^{p-1}\(\frac{j+cj^2+dj^3}{p}\)=0=\mu_p.$$
This shows that $[c,d]_p=0$. Moreover, if $E_p$ is an elliptic curve over $\F_p$, then by \cite[p. 148, Theorem 4.1]{Silverman} we know that $E_p$ is supersingular if $p$ divides the coefficient of $x^{p-1}$ of the polynomial $(dx^3+cx^2+x)^{\frac{p-1}{2}}$, i.e., $p\mid A_{c,d}(p)$.
In view of the above, (ii) holds. This completes the proof.\qed

We now turn to the proof of Theorem \ref{Thm. (c,1)p}. We first introduce the definition of a circulant matrix. Let $R$ be a commutative ring, and let $m$ be a positive integer. Let $a_0,a_1,\cdots,a_{m-1}\in R$. Then the {\it circulant matrix} of the $m$-tuple $(a_0,\cdots,a_{m-1})$ is defined by an $m\times m$ matrix over $R$ whose $(i,j)$-entry is $a_{i-j}$ where the indices are
cyclic modulo $m$. We also denote this matrix by $C(a_0,a_1,\cdots,a_{m-1})$.

The following lemma is the key element in our proof of Theorem \ref{Thm. (c,1)p}.
\begin{lemma}\label{Lem. The key lemma}
Let $R$ be a commutative ring. Let $m$ be a positive integer. Let $a_0,a_1,\cdots,a_{m-1}\in R$ such that
\begin{equation}\label{Eq. condition for ai}
a_i=a_{m-i}\ \ \ \ \ \text{for each}\ 1\le i\le m-1
\end{equation}
If $m$ is even, then there exists an element $u\in R$ such that
\begin{equation}\label{Eq. result of the key lemma}
\det C(a_0,a_1,\cdots,a_{m-1})=\(\sum_{i=0}^{m-1}a_i\)\(\sum_{i=0}^{m-1}(-1)^ia_i\)u^2
\end{equation}
If $m$ is odd, then there exists an element $v\in R$ such that
\begin{equation}\label{Eq. result2 of the key lemma}
\det C(a_0,a_1,\cdots,a_{m-1})=\(\sum_{i=0}^{m-1}a_i\)v^2
\end{equation}
\end{lemma}
\begin{proof}
We only prove the case that $m$ is even. With essentially the same method, one can also verify the case that $m$ is odd. Assume now that $m$ is even. Consider the polynomial ring $\Z[X_0,X_1,\cdots,X_{m/2}]$ and set
\begin{equation}\label{Eq. Xi=Xm-i in the key lemma}
X_{i}=X_{m-i}\ \ \ \ \ \text{for each}\ m/2<i<m.
\end{equation}
Consider the circulant matrix $C(X_0,X_1,\cdots,X_{m-1})$. We shall show that
\begin{equation}\label{Eq. polynomial version in the key lemma}
\det C(X_0,X_1,\cdots,X_{m-1})=\(\sum_{i=0}^{m-1}X_i\)\(\sum_{i=0}^{m-1}(-1)^iX_i\)P^2
\end{equation}
for some polynomial $P\in\Z[X_0,\cdots,X_{m/2}]$. Once this is proved, by substituting $a_0,a_1,\cdots,a_{m/2}$ for $X_0,X_1,\cdots,X_{m/2}$ our desired result will follow. Thus, it remains to prove (\ref{Eq. polynomial version in the key lemma}).

Let $\zeta_m\in\C$ be a primitive $m$-th root of unity, and let $S$ denote the ring $\Z[\zeta_m]$. The classical formula for the determinant of a circulant matrix (cf. \cite[Example 7.2]{Conrad} or \cite[Theorem 6]{Kra}) yields
$$\det C(X_0,X_1,\cdots,X_{m-1})=\prod_{j=0}^{m-1}\sum_{i=0}^{m-1}\zeta_m^{ji}X_i.$$
If we set
$$R_j:=\sum_{i=0}^{m-1}\zeta_{m}^{ji}X_i\ \ \ \ \ \text{for each}\ 1\le j\le m-1,$$
then this rewrites as
\begin{equation}\label{Eq. Rj in the key lemma}
\det C(X_0,X_1,\cdots,X_{m-1})=\prod_{j=0}^{m-1}R_j.
\end{equation}
The polynomials $R_0,\cdots,R_{m-1}$ lie in $\Z[X_0,\cdots,X_{m/2}]$. By (\ref{Eq. Xi=Xm-i in the key lemma}) we obtain that
\begin{equation}\label{Eq. Rj=Rm-j in the key lemma}
R_{m-j}=R_j\ \ \ \ \ \text{for each}\ 1\le j\le m-1.
\end{equation}
Let ${\rm G}(\Q(\zeta_m)/\Q)$ denote the Galois group of the Galois extension $\Q(\zeta_m)/\Q$. Given any $\sigma_g\in{\rm G}(\Q(\zeta_m)/\Q)$ with $(g,m)=1$ which sends $\zeta_m$ to $\zeta_m^g$. Note that $\sigma_g(S)=S$. Thus $\sigma_g$ acts on the polynomial ring $S[X_0,X_1,\cdots,X_{m/2}]$ (by acting on each coefficient). Considering this $g$, and letting the indices on the $R_j$ run cyclically modulo $m$, we thus see that $$\sigma_g(R_j)=R_{gj}\ \ \ \ \ \text{for each}\ 0\le j\le m-1.$$
We now claim that $\sigma_g$ permutes the $m/2-1$ polynomials $R_1,R_2,\cdots,R_{m/2-1}$. In fact, let $\pi_g$ be the map from $\{1,2,\cdots,m/2-1\}$ to $\{1,2,\cdots,m/2-1\}$ which sends each $j$ to the unique element $\pi_g(j)\in\{1,2,\cdots,m/2-1\}$ satisfying
$$\pi_g(j)\equiv\pm gj\pmod m.$$
Clearly $\pi_g$ is a permutation of $\{1,2,\cdots, m/2-1\}$. Combining this with (\ref{Eq. Rj=Rm-j in the key lemma}), we thus have $\sigma_g(R_j)=R_{\pi_g(j)}$ for each $1\le j\le m/2-1$ and hence our claim holds. By this $\sigma_g$ fixes $\prod_{j=1}^{m/2-1}R_j$. In other words, each coefficient of the polynomial $\prod_{j=1}^{m/2-1}R_j$ is fixed by $\sigma_g$.

In view of the above, by the Galois theory each coefficient of the polynomial $\prod_{j=1}^{m/2-1}R_j$ belongs to $\Q$. Moreover, as $\prod_{j=1}^{m/2-1}R_j\in S[X_0,\cdots,X_{m/2}]$, each coefficient of this polynomial is also an algebraic integer. This entails that each coefficient is a rational integer. In other words, the polynomial $\prod_{j=1}^{m/2-1}R_j\in\Z[X_0,\cdots,X_{m/2}]$.

Now by (\ref{Eq. Rj=Rm-j in the key lemma}) the equality (\ref{Eq. Rj in the key lemma}) becomes
\begin{align*}
\det C(X_0,\cdots,X_{m-1})=R_0R_{m/2}(R_1R_2\cdots R_{m/2-1})^2.
\end{align*}
Noting that
$$R_0=\sum_{i=0}^{m-1}X_i,\ R_{m/2}=\sum_{i=0}^{m-1}(-1)^iX_i,$$
and that $\prod_{j=1}^{m/2-1}R_j\in\Z[X_0,\cdots,X_{m/2}]$, one can easily verify that (\ref{Eq. polynomial version in the key lemma}) holds. This completes the proof.
\end{proof}

We also need the following known result in linear algebra. \begin{lemma}\label{Lem. eigenvalues}
Let $M$ be an $m\times m$ complex matrix. Let $\mu_1,\cdots,\mu_m$ be complex numbers, and let ${\bf u}_1,\cdots, {\bf u}_m$ be $m$-dimensional column vectors. If $M{\bf u}_k=\mu_k{\bf u}_k$ for each $1\le k\le m$ and ${\bf u}_1,\cdots, {\bf u}_m$ are linearly independent, then $\mu_1,\cdots,\mu_m$ are exactly all the eigenvalues of $M$ (counting multiplicity).
\end{lemma}
In addition, we let $\chi(\Z/p\Z)$ be the group of all multiplicative characters of the finite field $\Z/p\Z=\F_p$, and let $\chi_p$ be a generator of $\chi(\Z/p\Z)$. We now prove our theorem.

{\bf Proof of Theorem \ref{Thm. (c,1)p}.} (i) If $c\equiv\pm2\pmod p$, then clearly we have
$(c,1)_p=\det[(\frac{(i+cj/2)^2}{p})]_{1\le i,j\le p-1}$. This clearly is equal to
$$\sgn(\pi_p(-c/2))\times\left| \begin{array}{cccccccc}
 0  & 1  & \ldots  & 1  \\
 1 & 0 &  \ldots  &  1\\
 \vdots & \vdots & \ddots  & \vdots    \\
 1& 1 & \ldots &  0\\
 \end{array} \right|=\(\frac{-2c}{p}\)\times(2-p).$$
This implies (i).

(ii) We first let
$$N_p:=\bigg[\(\frac{i^2+cij+j^2}{p}\)\bigg]_{1\le i,j\le p-1}.$$
For each $1\le k\le p-1$, we let
$$\lambda_k:=\sum_{j=1}^{p-1}\(\frac{1+cj+j^2}{p}\)\chi_p^k(j).$$
We claim that the numbers $\lambda_1,\cdots,\lambda_{p-1}$ are exactly all the eigenvalues of $N_p$ (counting multiplicity). In fact, for any $1\le i,k\le p-1$ we have
\begin{align*}
\sum_{j=1}^{p-1}\(\frac{i^2+cij+j^2}{p}\)\chi_p^{k}(j)
&=\sum_{j=1}^{p-1}\(\frac{1+cj/i+j^2/i^2}{p}\)\chi_p^k(j/i)\chi_p^k(i)\\
&=\sum_{j=1}^{p-1}\(\frac{1+cj+j^2}{p}\)\chi_p^k(j)\chi_p^k(i)=\lambda_k\chi_p^k(i).
\end{align*}
This implies that for each $1\le k\le p-1$, we have
$$N_p{\bf v}_k=\lambda_k{\bf v}_k,$$
where
$${\bf v}_k:=(\chi_p^k(1),\chi_p^k(2),\cdots,\chi_p^k(p-1))^T.$$
Noting that
$$\left| \begin{array}{cccccccc}
\chi_p^1(1) & \chi_p^2(1) & \ldots  & \chi_p^{p-1}(1)  \\
\chi_p^1(2) & \chi_p^2(2) &  \ldots  &  \chi_p^{p-1}(2)\\
\vdots & \vdots & \ddots  & \vdots    \\
\chi_p^1(p-1)& \chi_p^2(p-1) & \ldots &  \chi_p^{p-1}(p-1)\\
\end{array} \right|^2=\prod_{0<i<j<p}\(\chi_p(j)-\chi_p(i)\)^2\ne0,$$
we therefore obtain that ${\bf v}_1,\cdots,{\bf v}_{p-1}$ are linearly independent. By Lemma \ref{Lem. eigenvalues} our claim holds. In addition, as $N_p$ is a real symmetric matrix, all the eigenvalues $\lambda_k$ are real. Hence for any $1\le k\le (p-3)/2$ we have $\lambda_k=\overline{\lambda_{k}}=\lambda_{p-1-k}$, where $\overline{\lambda_k}$ denotes the complex conjugation of $\lambda_k$.
When $k=p-1$, as $c\not\equiv\pm2\pmod p$, by \cite[p. 63, Exercise 8]{classical} it is easy to see that
\begin{equation}\label{Eq. E in Thm. (c,1)p}
\lambda_{p-1}=\sum_{j=1}^{p-1}\(\frac{1+cj+j^2}{p}\)=-2.
\end{equation}
When $k=\frac{p-1}{2}$, then
\begin{equation}\label{Eq. F in Thm. (c,1)p}
\lambda_{\frac{p-1}{2}}=\sum_{j=1}^{p-1}\(\frac{1+cj+j^2}{p}\)\(\frac{j}{p}\)=-a_p(c).
\end{equation}
Hence we have
\begin{equation}\label{Eq. G in Thm. (c,1)p}
\det N_p=\lambda_{p-1}\lambda_{\frac{p-1}{2}}\prod_{1\le k\le (p-3)/2}\lambda_k^2.
\end{equation}
We now apply Lemma \ref{Lem. The key lemma} to our remaining proof. Fix a primitive root $\xi$ modulo $p$. Since the sequence $\xi^1\ {\rm mod}\ p,\xi^2\ {\rm mod}\ p,\cdots,\xi^{p-1}\ {\rm mod}\ p$ is a permutation of the sequence $1\ {\rm mod}\ p, 2\ {\rm mod}\ p,\cdots, p-1\ {\rm mod}\ p$, one can verify that
\begin{align*}
\det N_p&=\det\bigg[\(\frac{\xi^{2i}+c\xi^{i+j}+\xi^{2j}}{p}\)\bigg]_{1\le i,j\le p-1}\\
&=\det\bigg[\(\frac{\xi^{2(i-j)}+c\xi^{i-j}+1}{p}\)\bigg]_{1\le i,j\le p-1}\\
&=\det C(\alpha_0,\alpha_1,\cdots,\alpha_{p-2}),
\end{align*}
where $\alpha_i=(\frac{\xi^{2i}+c\xi^i+1}{p})$ for any $0\le i\le p-2$. It is easy to verify that $\alpha_0,\cdots,\alpha_{p-2}$ satisfy the condition (\ref{Eq. condition for ai}) of Lemma \ref{Lem. The key lemma}. Hence
\begin{equation}\label{Eq. in the proof of (c,1)p circulant matrix}
\det C(\alpha_0,\alpha_1,\cdots,\alpha_{p-2})
=\(\sum_{i=0}^{p-2}\alpha_i\)\(\sum_{i=0}^{p-2}(-1)^i\alpha_i\)u_p(c)^2
\end{equation}
for some $u_p(c)\in\Z$. Note that
\begin{equation}\label{Eq. in the proof of (c,1)p sum1}
\sum_{i=0}^{p-2}\alpha_i=\sum_{j=1}^{p-1}\(\frac{j^2+cj+1}{p}\)=\lambda_{p-1},
\end{equation}
and that
\begin{equation}\label{Eq. in the proof of (c,1)p sum2}
\sum_{i=0}^{p-2}(-1)^i\alpha_i
=\sum_{i=0}^{p-2}\(\frac{\xi^i}{p}\)\alpha_i
=\sum_{j=1}^{p-1}\(\frac{1+cj+j^2}{p}\)\(\frac{j}{p}\)=\lambda_{(p-1)/2}.
\end{equation}
In view of (\ref{Eq. E in Thm. (c,1)p})-(\ref{Eq. in the proof of (c,1)p sum2}), we obtain that
$\det N_p=2a_c(p)x_c(p)^2$, where $x_c(p)=\pm u_p(c)=\pm\prod_{k=1}^{(p-3)/2}\lambda_k$ is an integer. This completes the proof. \qed

Now we turn to Corollary \ref{Cor. squares and determinants}.

{\bf Proof of Corollary \ref{Cor. squares and determinants}.} (i) We first consider $(3,2)_p$. If $p\equiv5\pmod8$, then by \cite[Theorem 1.3(i)]{Sun2019det} we have $(3,2)_p=0$. Hence we just need handle the case $p\equiv1\pmod8$. Recall that we write $p=x_a^2+4y_a^2$ with $x_a\equiv1\pmod4$.
Suppose $2\equiv f_p^2\pmod p$. Then by (\ref{Eq. introduction}) we have
\begin{equation}\label{Eq. (3,2)p transfer}
(3,2)_p=\(\frac{f_p}{p}\)\times(g_p,1)_p,
\end{equation}
where $g_p\in\Z$ with $g_p\equiv 3/f_p\pmod p$. Now we consider the elliptic curve
$$E_{g_p}(\F_p):=\{(x,y)\in\F_p\times\F_p: y^2=x^3+g_px^2+x\}\cup\{\infty\}.$$
Clearly $\#E_{g_p}(\F_p)$ is equal to
\begin{align*}
p+1-a_p(g_p)&=p+1+\sum_{j=0}^{p-1}\(\frac{j^2+g_pj+1}{p}\)\(\frac{j}{p}\)\\
&=p+1+\sum_{j=0}^{p-1}\(\frac{f_pj}{p}\)\(\frac{1+g_pf_pj+(f_pj)^2}{p}\)\\
&=p+1+\(\frac{f_p}{p}\)\sum_{j=0}^{p-1}\(\frac{j}{p}\)\(\frac{1+3j+2j^2}{p}\)\\
&=p+1-2x_a\(\frac{f_p}{p}\)
\end{align*}
The last equality follows from \cite[Lemma 3.3(ii)]{Dimitry}. Combining this with Theorem \ref{Thm. (c,1)p}(ii) and (\ref{Eq. (3,2)p transfer}), the number $(3,2)_p/x_a$ is a square.

(ii) Let $p\equiv1\pmod8$ be a prime with $p=x_b^2+2y_b^2$ $(x_b\equiv 1\pmod4)$. As before, we set $2\equiv t_p^2\pmod p$. Then by (\ref{Eq. introduction}) we clearly have
\begin{equation}\label{Eq. (4,2)p transfer}
(4,2)_p=(8,8)_p=\(\frac{t_p}{p}\)\times(e_p,1)_p,
\end{equation}
where $e_p\in\Z$ with $4/t_p\equiv e_p\pmod p$. Consider the curve
$$E_{e_p}(\F_p):=\{(x,y)\in\F_p\times\F_p: y^2=x^3+e_px^2+x\}\cup\{\infty\}.$$
Then clearly $\#E_{e_p}(\F_p)$ is equal to
\begin{align*}
p+1-a_p(e_p)&=p+1+\sum_{j=0}^{p-1}\(\frac{j^2+e_pj+1}{p}\)\(\frac{j}{p}\)\\
&=p+1+\sum_{j=0}^{p-1}\(\frac{(t_pj)^2+e_pt_pj+1}{p}\)\(\frac{t_pj}{p}\)\\
&=p+1+\(\frac{t_p}{p}\)\sum_{j=0}^{p-1}\(\frac{2j^2+4j+1}{p}\)\(\frac{j}{p}\)\\
&=p+1+2x_b(-1)^{\frac{p-9}{8}}\(\frac{t_p}{p}\).
\end{align*}
The last equality follows from \cite[Lemma 3.3(i)]{Dimitry}. Combining this with Theorem \ref{Thm. (c,1)p}(ii) and (\ref{Eq. (4,2)p transfer}), the number $(-1)^{\frac{p-1}{8}}(4,2)_p/x_b$ is a square.

Now let $p\equiv1\pmod{12}$ be a prime. Write $p=x_c^2+3y_c^2$ with $x_c\equiv1\pmod3$. Let $l_p^2\equiv3\pmod p$. Then by (\ref{Eq. introduction}) we have
\begin{equation}\label{Eq. (3,3)p transfer}
(3,3)_p=\(\frac{l_p}{p}\)\times(n_p,1)_p,
\end{equation}
where $n_p\in\Z$ with $n_p\equiv 3/l_p\pmod p$. Consider the curve
$$E_{n_p}(\F_p):=\{(x,y)\in\F_p\times\F_p: y^2=x^3+n_px^2+x\}\cup\{\infty\}.$$
Then with the same method as above, $\#E_{n_p}(\F_p)$ is equal to
\begin{align*}
p+1-a_{p}(n_p)&=p+1+\sum_{j=0}^{p-1}\(\frac{j^2+n_pj+1}{p}\)\(\frac{j}{p}\)\\
&=p+1+\(\frac{l_p}{p}\)\sum_{j=0}^{p-1}\(\frac{3j^2+3j+1}{p}\)\(\frac{j}{p}\)\\
&=p+1-2x_c\(\frac{l_p}{p}\).
\end{align*}
The last equality follows from \cite[Lemma 3.3]{Dimitry}. Combining this with Theorem \ref{Thm. (c,1)p}(ii) and (\ref{Eq. (3,3)p transfer}), the number $(3,3)_p/x_c$ is a square.

(iii) Let $p\equiv1,9,25\pmod{28}$ and write $p=x_d^2+7y_d^2$ with $(\frac{x_d}{7})=1$. Suppose now $112\equiv m_p^2\pmod p$. Then by (\ref{Eq. introduction}) we obtain
\begin{equation}\label{Eq. (21,112)p transfer}
(21,112)_p=\(\frac{m_p}{p}\)\times(r_p,1)_p,
\end{equation}
where $r_p\equiv 21/m_p\pmod p$. Consider the curve
$$E_{r_p}(\F_p):=\{(x,y)\in\F_p\times\F_p: y^2=x^3+r_px^2+x\}\cup\{\infty\}.$$
Then with the same method as above, $\#E_{r_p}(\F_p)$ is equal to
\begin{align*}
p+1-a_{p}(r_p)&=p+1+\sum_{j=0}^{p-1}\(\frac{j^2+r_pj+1}{p}\)\(\frac{j}{p}\)\\
&=p+1+\(\frac{m_p}{p}\)\sum_{j=0}^{p-1}\(\frac{112j^2+21j+1}{p}\)\(\frac{j}{p}\)\\
&=p+1-2x_d\(\frac{m_p}{p}\).
\end{align*}
The last equality follows from \cite[Lemma 3.3]{Dimitry}. Combining this with Theorem \ref{Thm. (c,1)p}(ii) and (\ref{Eq. (21,112)p transfer}), the number $(21,122)_p/x_d$ is a square.

In view of the above, we complete the proof.\qed
\maketitle
\section{Proof of Theorem \ref{Thm. biquadratic}}
\setcounter{lemma}{0}
\setcounter{theorem}{0}
\setcounter{corollary}{0}
\setcounter{remark}{0}
\setcounter{equation}{0}
\setcounter{conjecture}{0}
We first give the definition of Jacobsthal sums. Readers may consult \cite[Chapter 6]{Berndt}. For any positive integer $k$, the Jacobsthal sums $\phi_k(1)$ and $\psi_k(1)$ are defined by
\begin{align*}
\phi_k(1)&=\sum_{x=1}^{p-1}\(\frac{x}{p}\)\(\frac{x^k+1}{p}\),
\\\psi_k(1)&=\sum_{x=1}^{p-1}\(\frac{x^k+1}{p}\).
\end{align*}
We need the following result involving Jacobsthal sums.
\begin{lemma}\label{Lem. Jacobsthal sums}
{\rm (i)} {\rm(\cite[Theorem 6.1.13]{Berndt})} For any positive integer $k$, we have
$$\psi_{2k}(1)=\phi_{k}(1)+\psi_k(1).$$

{\rm (ii)} {\rm(\cite[Theorem 6.2.9]{Berndt})} Let $p\equiv1\pmod4$ be a prime with $p=x_1^2+4y_1^2$ $(x_1,y_1\in\Z$ and $x_1\equiv1\pmod4)$. Then we have
$$\phi_2(1)=-2x_1.$$

{\rm (iii)} {\rm(\cite[Theorem 6.2.3]{Berndt})} Let $p\equiv1\pmod8$ be a prime and let $p=x_2^2+2y_2^2$ with $x_2,y_2\in\Z$ and $x_2\equiv1\pmod4$. Then we have
$$\phi_4(1)=-4x_2(-1)^{\frac{p-1}{8}}.$$
\end{lemma}

{\bf Proof of Theorem \ref{Thm. biquadratic}.} Throughout this proof, we set $n=\frac{p-1}{4}$ and set
$$L_p=\bigg[\(\frac{a_i+a_j}{p}\)\bigg]_{1\le i,j\le n}.$$
Recall that $\chi_p$ denotes a generator of the character group $\chi(\Z/p\Z)$. For $1\le k\le n$, we set
$$\theta_k:=\sum_{j=1}^n\(\frac{1+a_j}{p}\)\chi_p^k(a_j).$$
We claim that $\theta_k$ are exactly all the eigenvalues of $L_p$. In fact, for any $1\le i\le n$, we have
\begin{align*}
\sum_{j=1}^n\(\frac{a_i+a_j}{p}\)\chi_p^k(a_j)
&=\sum_{j=1}^n\(\frac{1+a_j/a_i}{p}\)\chi_p^k(a_j/a_i)\chi_p^k(a_i)\\
&=\sum_{j=1}^n\(\frac{1+a_j}{p}\)\chi_p^k(a_j)\chi_p^k(a_i).
\end{align*}
This shows that for any $1\le k\le n$ we have
$$L_p{\bf w}_k=\theta_k{\bf w}_k,$$
where
$${\bf w}_k:=(\chi_p^k(a_1),\cdots,\chi_p^k(a_n))^T.$$
As in the proof of Theorem \ref{Thm. (c,1)p}, one can verify that our claim holds. In addition, since $L_p$ is a real symmetric matrix, each eigenvalue $\theta_k$ is real. Hence it is easy to verify that
\begin{equation}\label{Eq. in the proof of det Lp prod of theta k}
\det L_p=\begin{cases}\theta_n\Gamma&\mbox{if}\ p\equiv5\pmod8,\\ \theta_n\theta_{n/2}\Gamma&\mbox{if}\ p\equiv1\pmod8,\end{cases}
\end{equation}
where
$$\Gamma=\begin{cases}\prod_{k=1}^{n-1}\theta_k=\prod_{k=1}^{\frac{n-1}{2}}\theta_k^2&\mbox{if}\ p\equiv 5\pmod8,\\\prod_{k\ne n,n/2}\theta_k=\prod_{k=1}^{\frac{n}{2}-1}\theta_k^2&\mbox{if}\ p\equiv 1\pmod8.\end{cases}$$
Moreover, for each $1\le i\le n$, we choose an integer $b_i$ such that $b_i^4\equiv a_i\pmod p$. We first consider $\theta_n$. It is clear that
\begin{equation}\label{Eq. 2 in the proof of Thm. biquadratic}
\theta_n=\sum_{j=1}^n\(\frac{1+b_j^4}{p}\)=\frac{\psi_{4}(1)}{4}=\frac{-x_1-1}{2}.
\end{equation}
The last equality follows from Lemma \ref{Lem. Jacobsthal sums} and $\psi_2(1)=-2$. In the case $p\equiv1\pmod8$, the number $n$ is even. One can verify that
\begin{equation}\label{Eq. 3 in the proof of Thm. biquadratic}
\theta_{n/2}=\sum_{j=1}^n\(\frac{1+b_j^4}{p}\)\(\frac{b_j}{p}\)
=\frac{\phi_4(1)}{4}=-x_2(-1)^{\frac{p-1}{8}}.
\end{equation}
The last equality follows from Lemma \ref{Lem. Jacobsthal sums}. Fix a primitive root $\xi$ modulo $p$. One can verify that
\begin{align*}
\det L_p&=\det\bigg[\(\frac{\xi^{4i}+\xi^{4j}}{p}\)\bigg]_{0\le i,j\le n-1}\\
&=\det\bigg[\(\frac{\xi^{4(i-j)}+1}{p}\)\bigg]_{0\le i,j\le n-1}\\
&=\det C(\beta_0,\beta_1,\cdots,\beta_{n-1}),
\end{align*}
where $\beta_i=(\frac{\xi^{4i}+1}{p})$ for each $0\le i\le n-1$. Clearly $\beta_0,\cdots,\beta_{n-1}$ satisfy the condition (\ref{Eq. condition for ai}) in Lemma \ref{Lem. The key lemma}. Note that
\begin{equation}\label{Eq. sum of bi in the proof of det Lp}
\sum_{i=0}^{n-1}\beta_i=\sum_{j=1}^n\(\frac{1+a_j}{p}\)=\theta_n,
\end{equation}
and that when $n$ is even,
\begin{equation}\label{Eq. sums of bi 2 in the proof of det Lp}
\sum_{i=0}^{n-1}(-1)^i\beta_i=\(\frac{\xi^i}{p}\)\(\frac{1+\xi^{4i}}{p}\)
=\sum_{j=1}^{n}\(\frac{b_j}{p}\)\(\frac{1+b_j^4}{p}\)=\theta_{n/2}.
\end{equation}
Combining Lemma \ref{Lem. The key lemma} with (\ref{Eq. in the proof of det Lp prod of theta k})-(\ref{Eq. sums of bi 2 in the proof of det Lp}), we obtain that $\Gamma$ is an integral square. This implies our desired result. \qed

We now turn to the proof of our last theorem. We first need the following result.
\begin{lemma}\label{Lem. Jacobsthal sums2}
{\rm (i)} {\rm(\cite[Proposition 6.1.7, Theorem 6.2.10]{Berndt})} Let $p\equiv1\pmod3$, and let $p=x_3^2+3y_3^2$ with $x_3\equiv1\pmod3$. Then we have
$$\phi_3(1)=\psi_3(1)=-1-2x_3.$$

{\rm (ii)} {\rm(\cite[Theorem 6.2.5]{Berndt})} Let $p\equiv1\pmod{12}$ be a prime, and let $p=x_4^2+4y_4^2$ with $x_4,y_4\in\Z$ and $x_4\equiv(\frac{2}{p})\pmod4$. Then we have
$$\phi_6(1)=\begin{cases}2x_4(-1)^{\frac{p-1}{4}}&\mbox{if}\ 3\mid x_4,\\-6x_4(-1)^{\frac{p-1}{4}}&\mbox{if}\ 3\nmid x_4.\end{cases}$$
\end{lemma}
Now we are in a position to prove our theorem.

{\bf Proof of Theorem \ref{Thm. six}}. Throughout this proof, we let $m=\frac{p-1}{6}$ and let
$$R_p:=\bigg[\(\frac{c_i+c_j}{p}\)\bigg]_{1\le i,j\le m}.$$
We also let $\chi_p$ denote a generator of the character group $\chi(\Z/p\Z)$.
For $1\le k\le m$, we let
$$\gamma_k:=\sum_{j=1}^m\(\frac{1+c_j}{p}\)\chi_p^k(c_j).$$
We claim that $\gamma_k$ are exactly all the eigenvalues of $R_p$. In fact, for any $1\le i\le m$ we have
\begin{align*}
\sum_{j=1}^m\(\frac{c_i+c_j}{p}\)\chi_p^k(c_j)
&=\sum_{j=1}^m\(\frac{1+c_j/c_i}{p}\)\chi_p^k(c_j/c_i)\chi_p^k(c_i)\\
&=\sum_{j=1}^m\(\frac{1+c_j}{p}\)\chi_p^k(c_j)\chi_p^k(c_i).
\end{align*}
This implies that
$$R_p{\bf z}_k=\gamma_k{\bf z}_k,$$
where
$${\bf z}_k=(\chi_p^k(c_1),\cdots,\chi_p^k(c_m))^T.$$
As in the proof of Theorem \ref{Thm. (c,1)p}, one can verify that our claim holds. In addition, as $R_p$ is a real symmetric matrix, each $\gamma_k$ is real. Hence it is easy to see that
\begin{equation}\label{Eq. in the proof of det Rp prod of gamma k}
\det R_p=\begin{cases}\gamma_m\cdot\Lambda&\mbox{if}\ p\equiv 7\pmod{12},\\ \gamma_m\gamma_{m/2}\cdot\Lambda&\mbox{if}\ p\equiv 1\pmod{12},\end{cases}
\end{equation}
where
$$\Lambda=\begin{cases}\prod_{j=1}^{m-1}\gamma_k=\prod_{j=1}^{\frac{m-1}{2}}\gamma_k^2&\mbox{if}\ p\equiv 7\pmod{12},\\\prod_{j\ne m,m/2}\gamma_k=\prod_{j=1}^{\frac{m}{2}-1}\gamma_k^2&\mbox{if}\ p\equiv 1\pmod{12}.\end{cases}$$
For any $1\le i\le m$, we choose an integer $d_i$ such that $d_i^6\equiv c_i\pmod p.$ It is easy to verify that
\begin{equation}\label{Eq. 2 in the proof of Thm. six}
\gamma_m=\sum_{j=1}^m\(\frac{1+d_j^6}{p}\)=\frac{\psi_6(1)}{6}=\frac{-1-2x_3}{3}.
\end{equation}
The last equality follows from Lemma \ref{Lem. Jacobsthal sums} and Lemma \ref{Lem. Jacobsthal sums2}. Moreover, when $p\equiv1\pmod{12}$, we have $2\mid m$. One can verify that
\begin{equation}\label{Eq. 3 in the proof of Thm. six}
\gamma_{m/2}=\sum_{j=1}^m\(\frac{1+d_j^6}{p}\)\(\frac{d_j}{p}\)
=\frac{\phi_6(1)}{6}=\begin{cases}\frac{1}{3}x_4(-1)^{\frac{p-1}{4}}&\mbox{if}\ 3\mid x_4,\\-x_4(-1)^{\frac{p-1}{4}}&\mbox{if}\ 3\nmid x_4.\end{cases}
\end{equation}
Fix a primitive root $\xi$ modulo $p$. We have
\begin{align*}
\det R_p&=\det\bigg[\(\frac{\xi^{6i}+\xi^{6j}}{p}\)\bigg]_{0\le i,j\le m-1}\\
&=\det\bigg[\(\frac{\xi^{6(i-j)}+1}{p}\)\bigg]_{0\le i,j\le m-1}\\
&=\det C(t_0,t_1,\cdots,t_{m-1}),
\end{align*}
where $t_i=(\frac{\xi^{6i}+1}{p})$ for any $0\le i\le m-1$. Clearly $t_0,\cdots,t_{m-1}$ satisfy the condition (\ref{Eq. condition for ai}) in Lemma \ref{Lem. The key lemma}. Observe that
\begin{equation}\label{Eq. in the proof of det Rp sum of ti}
\sum_{i=0}^{m-1}t_i=\sum_{j=1}^{m}\(\frac{1+c_j}{p}\)=\gamma_m,
\end{equation}
and that when $m$ is even,
\begin{equation}\label{Eq. in the proof of det Rp sum of ti2}
\sum_{i=0}^{m-1}(-1)^it_i=\sum_{i=0}^{m-1}\(\frac{\xi^{i}}{p}\)\(\frac{\xi^{6i}+1}{p}\)
=\sum_{j=1}^m\(\frac{1+d_j^6}{p}\)\(\frac{d_j}{p}\)=\gamma_{m/2}.
\end{equation}
Applying Lemma \ref{Lem. The key lemma} and using (\ref{Eq. in the proof of det Rp prod of gamma k})-(\ref{Eq. in the proof of det Rp sum of ti2}), one can obtain that $\Lambda$ is an integral square. This implies our desired result. \qed

\Ack\ The author would like to thank the two anonymous referees for their careful reading and indispensable suggestions. The author also thank Prof. Hao Pan, Dr. Yue-Feng She and Dr. Li-Yuan Wang for their encouragement.

Prof. Robin Chapman unexpectedly passed away (aged just 57) on October 18, 2020. We dedicate this paper to his memory.

This research was supported by the National Natural Science Foundation of China (Grant No. 11971222) and NUPTSF (Grant No. NY220159).

\end{document}